\pdfoutput=1
\RequirePackage{ifpdf}
\ifpdf % We~are running pdfTeX in pdf mode
\documentclass[pdftex]{sigma}
\else
\documentclass{sigma}
\fi

\numberwithin{equation}{section}

\newtheorem{Theorem}{Theorem}[section]
\newtheorem*{Theorem*}{Theorem}
\newtheorem{Corollary}[Theorem]{Corollary}
\newtheorem{Lemma}[Theorem]{Lemma}

 { \theoremstyle{definition}
\newtheorem{Definition}[Theorem]{Definition}

 }

\newcommand{\fr}{\frac}

\newmuskip\pFqskip
\mathchardef\pFcomma=\mathcode`, % keep a copy of the comma
\newcommand*\pFq[5]{%
 \begingroup
 \begingroup\lccode`~=`,
 \lowercase{\endgroup\def~}{\pFcomma\mkern\pFqskip}%
 \mathcode`,=\string"8000
 {}_{#1}\phi_{#2}\biggl[\genfrac..{0pt}{}{#3}{#4};#5\biggr]%
 \endgroup
 }
 \newmuskip\pGqskip
\mathchardef\pGcomma=\mathcode`, % keep a copy of the comma

\begin{document}
\allowdisplaybreaks

\newcommand{\arXivNumber}{2503.10890}

\renewcommand{\thefootnote}{}

\renewcommand{\PaperNumber}{056}

\FirstPageHeading

\ShortArticleName{Positive Weighted Partitions Generated by Double Series}

\ArticleName{Positive Weighted Partitions Generated\\ by Double Series\footnote{This paper is a~contribution to the Special Issue on Basic Hypergeometric Series Associated with Root Systems and Applications in honor of Stephen C.~Milne's 75th birthday. The~full collection is available at \href{https://www.emis.de/journals/SIGMA/Milne.html}{https://www.emis.de/journals/SIGMA/Milne.html}}}

\Author{George E. ANDREWS~$^{\rm a}$ and Mohamed EL BACHRAOUI~$^{\rm b}$}

\AuthorNameForHeading{G.E.~Andrews and M.~El Bachraoui}

\Address{$^{\rm a)}$~The Pennsylvania State University, University Park, Pennsylvania 16802, USA}
\EmailD{\href{mailto:andrews@math.psu.edu}{andrews@math.psu.edu}}

\Address{$^{\rm b)}$~United Arab Emirates University, PO Box 15551, Al-Ain, United Arab Emirates}
\EmailD{\href{mailto:melbachraoui@uaeu.ac.ae}{melbachraoui@uaeu.ac.ae}}

\ArticleDates{Received March 14, 2025, in final form July 04, 2025; Published online July 12, 2025}

\Abstract{We investigate some weighted integer partitions whose generating functions are double-series. We will establish closed formulas for these $q$-double series and deduce that their coefficients are non-negative. This leads to inequalities among integer partitions.}


\Classification{11P81; 05A17; 11D09}

\begin{flushright}
\textit{In honor of Steve Milne's 75th birthday}
\end{flushright}

\renewcommand{\thefootnote}{\arabic{footnote}}
\setcounter{footnote}{0}

\section{Introduction}
Throughout, let $q$ denote a complex number satisfying $|q|<1$
and let $m$ and $n$ denote nonnegative integers.
We adopt the following standard notation from the theory of $q$-series~\cite{Andrews, Gasper/Rahman}
\begin{gather*}
(a;q)_0 = 1,\qquad v (a;q)_n = \prod_{j=0}^{n-1} \bigl(1-aq^j\bigr),\qquad
(a;q)_{\infty} = \prod_{j=0}^{\infty} \bigl(1-aq^j\bigr),
\\
(a_1,\ldots,a_k;q)_n = \prod_{j=1}^k (a_j;q)_n,\qquad \text{and}\qquad
(a_1,\ldots,a_k;q)_{\infty} = \prod_{j=1}^k (a_j;q)_{\infty}.
\end{gather*}
We shall need the following basic facts
\begin{equation}
(a;q)_{n+m} = (a;q)_{m} (aq^{m};q)_n\qquad \text{and}\qquad
(a;q)_{\infty} = (a;q)_n (aq^n;q)_{\infty}.\label{basic}
\end{equation}
In this paper, we consider certain $q$-double series in one single variable
which turn out to be natural generating functions for weighted integer partitions.

Weighted integer partitions have been extensively studied in the past. A first systematic investigation of identities for weighted partitions is due to
Alladi~\cite{Alladi1997, Alladi1998}. For other references on weighted partitions and their applications,
see, for instance,~\cite{Garvan2018,Fokkink/Wang:1995}.

A power series $\sum_{n\geq 0} a_n q^n$ is called positive, written $\sum_{n\geq 0} a_n q^n \succeq 0$,
if $a_n \geq 0$ for any nonnegative integer $n$. Accordingly, we will write
$\sum_{n\geq 0} a_n q^n \succeq \sum_{n\geq 0} b_n q^n$ to mean that
$\sum_{n\geq 0} a_n q^n - \sum_{n\geq 0} b_n q^n \succeq 0$.
Positivity results for $q$-series have been intensively studied in the past to some extent in connection with
Borwein's famous positivity conjecture~\cite{Andrews1995}.
For more on this, see, for instance,~\cite{Berkovich/Warnaar:2005, Bressoud1996, Wang2022, Warnaar/Zudilin:2011}.
Positivity results for alternating sums have also received much attention in recent years, see
for example~\cite{Andrews/Merca:2012, Andrews/Merca:2018, Bachraoui2020, Guo/Zeng:2013}.

An important application of positivity results is the fact that positive series which are generating
functions for weighted partitions give rise to inequalities of integer partitions. About the
interplay between positive $q$-series and inequalities of integer partitions, we refer the reader
to~\cite{Andrews2013, Berkovich/Grizzell:2013, Berkovich/Grizzell:2014, Berkovich-Uncu, Laughlin2016}.

Our main goal in this work is to prove that certain $q$-double series are positive. As these series
turn out to be generating functions for weighted partitions, our results
yield inequalities of integer partitions.

The paper is organized as follows. In Section~\ref{sec results}, we introduce our series through the partitions they generate and we state our main
results. In Section~\ref{sec lemmas}, we collect the lemmas needed to prove the main theorems.
Sections~\ref{sec proof double-1}--\ref{sec proof double-3} are devoted to the proofs of the main results and their corollaries.
\section{Main results}\label{sec results}
\begin{Definition}\label{def-1}
For any positive integer $N$,
let $F_1(N)$ be the number of partitions of $N$,
where if the partition has $n$ ones then the largest part is $2n+2k+1$ for some $k$ and all parts $>1$
are in the interval $[2n+2,2n+2k+1]$, no even parts are repeated, and the partition is counted
with weight $(-1)^j$, where $j$ is the number of even parts.
Then we have
\[
\sum_{n\geq 0} F_1(n) q^n
=\sum_{k,n\geq 0} \fr{\bigl(q^{2n+2};q^2\bigr)_k}{\bigl(q^{2n+3};q^2\bigr)_k}q^{2k+3n+1}.
\]
\end{Definition}
We now state our first main result.
\begin{Theorem}\label{thm double-1}
We have
\[
\sum_{n\geq 0} F_1(n) q^n
=\fr{1}{\bigl(1-q^2\bigr)}\fr{\bigl(q^2;q^2\bigr)_\infty^2}{\bigl(q;q^2\bigr)_\infty^2}
-\fr{\bigl(1+q^2\bigr)}{(1-q)\bigl(1-q^3\bigr)}.
\]
\end{Theorem}
With the help of Theorem~\ref{thm double-1}, we will derive the following positivity result.
\begin{Corollary}\label{cor positive-1}
There holds
$
\sum_{n\geq 0} F_1(n) q^n \succeq 0$.
\end{Corollary}
We now introduce our second integer partitions.
\begin{Definition}\label{def-2}
For any positive integer $N$,
let $F_2(N)$ be the number of partitions of $N$ such that for each $j=0,1,2$ satisfying
$3\mid 2k+j$ for some $k$, there are $n+j-2$ ones and $(2k+j)/3$ threes, the remaining
parts lie in the set $\{2n+2\}\cup (2n+3,2n+2k+3]$, no even parts are repeated, and the partition is counted
with weight $(-1)^j$, where $j$ is the number of even parts.
Then we have
\[
\sum_{n\geq 0} F_2(n) q^n
=\sum_{k,n\geq 0} \fr{\bigl(q^{2n+2};q^2\bigr)_k}{\bigl(q^{2n+5};q^2\bigr)_k}q^{2k+n+2}.
\]
\end{Definition}
\begin{Theorem}\label{thm double-2}
We have
\[
\sum_{n\geq 0} F_2(n) q^n
=\fr{q\bigl(1-q^3\bigr)}{(1-q)\bigl(1-q^2\bigr)}\fr{\bigl(q^2;q^2\bigr)_\infty^2}{\bigl(q;q^2\bigr)_\infty^2}
-\fr{q\bigl(1+q^2\bigr)}{(1-q)^2}.
\]
\end{Theorem}
\begin{Corollary}\label{cor positive-2}
There holds
$
\sum_{n\geq 0} F_2(n) q^n \succeq 0$.
\end{Corollary}
We now deal with our third example of integer partitions.
\begin{Definition}\label{def-3}
For any positive integer $N$,
let $G(N)$ be the number of partitions of $N$,
where if the partition has $3n$ ones then the largest part is $2n+2k+2$ for some $k$ and all parts
are in the interval $[2n+1,2n+2k+2]$, no even parts are repeated, and the partition is counted
with weight $(-1)^j$, where $j$ is the number of even parts.
Then we have
\[
\sum_{n\geq 0} G(n) q^n
=\sum_{k,n\geq 0} \fr{\bigl(q^{2n+2};q^2\bigr)_k}{\bigl(q^{2n+1};q^2\bigr)_k}q^{2k+5n+2}.
\]
\end{Definition}
\begin{Theorem}\label{thm double-3}
We have
\[
\sum_{n\geq 0} G(n) q^n
=\fr{q^3}{(1+q)\bigl(1-q^3\bigr)}\fr{\bigl(q^2;q^2\bigr)_\infty^2}{\bigl(q;q^2\bigr)_\infty^2}
-\fr{q^2(1-q)\bigl(-1+q^3+q^4+q^5\bigr)}{\bigl(1-q^3\bigr)^2 \bigl(1-q^5\bigr)}.
\]
\end{Theorem}
\begin{Corollary}\label{cor positive-3}
There holds
$
\sum_{n\geq 0} G(n) q^n \succeq 0$.
\end{Corollary}
Our proofs for Corollaries~\ref{cor positive-1}, \ref{cor positive-2} and \ref{cor positive-3} on positive weighted partitions are all analytic.
Obviously, each of these three corollaries is equivalent to an inequality of integer partitions. So, it is natural to
ask for injective proofs for these inequalities.

\section{Preliminary lemmas}\label{sec lemmas}

In this section, we collect several lemmas which we will need to prove our main results.
To simplify the presentation, we introduce the following sequences.
\begin{Definition}\label{def sequences}
For any positive integers $m$ and $n$, let
\begin{gather*}
\sum_{n\geq 0} A(m,n) q^n
=\sum_{n\geq 0} \fr{q^{mn}\bigl(q^{2n+2};q^2\bigr)_\infty \bigl(q^{2n+4};q^2\bigr)_\infty}{\bigl(q^{2n-1};q^2\bigr)_\infty \bigl(q^{2n+1};q^2\bigr)_\infty}, \\
\sum_{n\geq 0} A'(m,n) q^n =
\sum_{n\geq 0}\fr{q^{m(n+1)}\bigl(q^{2n+2};q^2\bigr)_\infty \bigl(q^{2n+6};q^2\bigr)_\infty}{\bigl(q^{2n+1};q^2\bigr)_\infty\bigl(q^{2n+3};q^2\bigr)_\infty}, \\
\sum_{n\geq 0} B(m,n) q^n
=\sum_{n\geq 0} \fr{q^{mn}\bigl(q^{2n+2};q^2\bigr)_\infty \bigl(q^{2n+4};q^2\bigr)_\infty}{\bigl(q^{2n-3};q^2\bigr)_\infty \bigl(q^{2n+3};q^2\bigr)_\infty}, \\
\sum_{n\geq 0} B'(m,n) q^n
=\sum_{n\geq 0}\fr{q^{m(n+1)}\bigl(q^{2n+2};q^2\bigr)_\infty \bigl(q^{2n+6};q^2\bigr)_\infty}{\bigl(q^{2n-1};q^2\bigr)_\infty\bigl(q^{2n+5};q^2\bigr)_\infty}.
\end{gather*}
\end{Definition}
\begin{Lemma}\label{lem A}
We have
\begin{equation}
\sum_{n\geq 0}A(2,n)q^n = \fr{-q\bigl(1+q^2\bigr)}{(1-q)^2 \bigl(1-q^3\bigr)}\label{A2}
\end{equation}
and for any positive integer $m\geq 2$,
\begin{equation}
\sum_{n\geq 0} A(2m,n) q^n
= \fr{-1}{(1-q)^2}\sum_{n\geq 0}\fr{q^{n+1}\bigl(q^{2n+2};q^2\bigr)_{m-1}}{\bigl(q^{2n+3};q^2\bigr)_{m-2}}.\label{A2m}
\end{equation}
\end{Lemma}
\begin{proof}
Throughout we will use Heine's transformations~\cite[Appendix~III, equations~(III.1)--(III.2)]{Gasper/Rahman}
\begin{align}
\pFq{2}{1}{a,b}{c}{q, z} &=\fr{(b,az;q)_\infty}{(c,z;q)_\infty}\pFq{2}{1}{c/b,z}{az}{q, b}, \label{Heine-1}\\
&= \fr{(c/b,bz;q)_\infty}{(c,z;q)_\infty}\pFq{2}{1}{abz/c,b}{bz}{q, c/b}, \label{Heine-2}
\end{align}
where
\[
\pFq{2}{1}{a,b}{c}{q, z} = \sum_{n\geq 0}\fr{(a;q)_n (b;q)_n}{(q;q)_n (c;q)_n} z^n.
\]
We will also use the $q$-binomial theorem~\cite[equation (2.2.1)]{Andrews},
\begin{equation}
\sum_{n\geq 0} \fr{(a;q)_n}{(q;q)_n}z^n = \fr{(az;q)_\infty}{(z;q)_\infty}\label{q-binom}
\end{equation}
and Ramanujan's $_1\psi_1$ summation formula~\cite[equation (5.2.1)]{Gasper/Rahman}
\begin{equation}
\sum_{n=-\infty}^{\infty}\fr{(a;q)_n}{(b;q)_n}z^n
=\fr{(q,b/a,az,q/(az);q)_\infty}{(b,q/a,z,b/(az);q)_\infty}.\label{psi-sum}
\end{equation}
We start with~\eqref{A2}. By~\eqref{basic} and~\eqref{Heine-2}, we obtain
\begin{align*}
\sum_{n\geq 0}A(2,n)q^{2n} &=\sum_{n\geq 0} \fr{q^{2n} \bigl(q^{2n+2},q^{2n+4};q^2\bigr)_\infty}{\bigl(q^{2n-1},q^{2n+1};q^2\bigr)_\infty} =\fr{\bigl(q^2,q^4;q^2\bigr)_\infty}{\bigl(q^{-1},q;q^2\bigr)_\infty} \sum_{n\geq 0}\fr{\bigl(q^{-1},q;q^2\bigr)_n q^{2n}}{\bigl(q^2,q^4;q^2\bigr)_n} \\
&=\fr{\bigl(q^2,q^4;q^2\bigr)_\infty}{\bigl(q^{-1},q;q^2\bigr)_\infty}\pFq{2}{1}{q^{-1},q}{q^{4}}{q^2, q^{2}} \\
&=\fr{\bigl(q^2,q^4;q^2\bigr)_\infty}{\bigl(q^{-1},q;q^2\bigr)_\infty} \fr{\bigl(q^3,q^3;q^2\bigr)_\infty}{\bigl(q^2,q^4;q^2\bigr)_\infty}
\pFq{2}{1}{q^{-2},q}{q^{3}}{q^2, q^{3}} \\
&=\fr{-q}{(1-q)^3} \left(1- \fr{q(1-q)\bigl(1-q^2\bigr)}{\bigl(1-q^2\bigr)\bigl(1-q^3\bigr)} \right)
=\fr{-q\bigl(1+q^2\bigr)}{(1-q)^2 \bigl(1-q^3\bigr)}.
\end{align*}
Now assume that $m$ is a positive integer such that $m\geq 2$.
We get by using~\eqref{basic} and~\eqref{Heine-1}
\begin{align*}
\sum_{n\geq 0} A(2m,n) q^n
={}&\sum_{n\geq 0} \fr{q^{2mn} \bigl(q^{2n+2},q^{2n+4};q^2\bigr)_\infty}{\bigl(q^{2n-1},q^{2n+1};q^2\bigr)_\infty}
=\fr{\bigl(q^2,q^4;q^2\bigr)_\infty}{\bigl(q^{-1},q;q^2\bigr)_\infty} \sum_{n\geq 0} \fr{\bigl(q^{-1},q;q^2\bigr)_n q^{2mn}}{\bigl(q^2,q^4;q^2\bigr)_n} \\
={}&\fr{\bigl(q^2,q^4;q^2\bigr)_\infty}{\bigl(q^{-1},q;q^2\bigr)_\infty}
\pFq{2}{1}{q^{-1},q}{q^{4}}{q^2, q^{2m}} \\
={}&\fr{\bigl(q^2,q^4;q^2\bigr)_\infty}{\bigl(q^{-1},q;q^2\bigr)_\infty}\fr{\bigl(q,q^{2m-1};q^2\bigr)_\infty}{\bigl(q^4,q^{2m};q^2\bigr)_\infty}
\pFq{2}{1}{q^{3},q^{2m}}{q^{2m-1}}{q^2, q} \\
={}&\fr{\bigl(q^2;q^2\bigr)_\infty}{\bigl(q^{-1};q^2\bigr)_\infty}\fr{\bigl(q^{2m-1};q^2\bigr)_\infty}{\bigl(q^{2m};q^2\bigr)_\infty}
\fr{\bigl(q^3,q^{2m};q^2\bigr)_\infty}{\bigl(q^2,q^{2m-1};q^2\bigr)_\infty}\sum_{n\geq 0}\fr{q^n \bigl(q^{2n+2}, q^{2n+2m-1};q^2\bigr)_\infty}{\bigl(q^{2n+3},q^{2n+2m};q^2\bigr)_\infty} \\
={}&\fr{-q}{(1-q)^2}\sum_{n\geq 0}\fr{q^n \bigl(q^{2n+2};q^2\bigr)_{m-1}}{\bigl(q^{2n+3};q^2\bigr)_{m-2}}
=\fr{-1}{(1-q)^2}\sum_{n\geq 1}\fr{q^n \bigl(q^{2n};q^2\bigr)_{m-1}}{\bigl(q^{2n+1};q^2\bigr)_{m-2}},
\end{align*}
which confirms~\eqref{A2m}.
\end{proof}
\begin{Lemma}\label{lem A'}
For any positive integer, we have
\begin{align}
\sum_{n\geq 0} A'(2m,n) q^n
&=\fr{1}{1-q^3}\sum_{n\geq 0}\fr{q^{n+2m}\bigl(q^{2n+2};q^2\bigr)_{m-1}}{\bigl(q^{2n+5};q^2\bigr)_{m-1}} \label{A'-1} \\
&=\fr{1}{1-q}\sum_{n\geq 0}\fr{q^{3n+2m}\bigl(q^{2n+2};q^2\bigr)_{m-1}}{\bigl(q^{2n+3};q^2\bigr)_{m-1}} \label{A'-2}.
\end{align}
\end{Lemma}
\begin{proof}
We start with~\eqref{A'-1}. We have by~\eqref{basic}
\begin{align*}
\sum_{n\geq 0} A'(2m,n) q^n
&=\sum_{n\geq 0}\fr{q^{2m(n+1)}\bigl(q^{2n+2};q^2\bigr)_\infty \bigl(q^{2n+6};q^2\bigr)_\infty}{\bigl(q^{2n+1};q^2\bigr)_\infty\bigl(q^{2n+3};q^2\bigr)_\infty} \\
&=q^{2m} \fr{\bigl(q^2,q^6;q^2\bigr)_\infty}{\bigl(q,q^3;q^2\bigr)_\infty}
\sum_{n\geq 0}\fr{q^{2m}\bigl(q,q^3;q^2\bigr)_n}{\bigl(q^2,q^6;q^2\bigr)_n} \\
&=q^{2m} \fr{\bigl(q^2,q^6;q^2\bigr)_\infty}{\bigl(q,q^3;q^2\bigr)_\infty}
\pFq{2}{1}{q,q^3}{q^6}{q^2, q^{2m}} \\
&=q^{2m} \fr{\bigl(q^2,q^6;q^2\bigr)_\infty}{\bigl(q,q^3;q^2\bigr)_\infty}
\fr{\bigl(q,q^{2m+3};q^2\bigr)_\infty}{\bigl(q^6,q^{2m};q^2\bigr)_\infty}
\pFq{2}{1}{q^5,q^{2m}}{q^{2m+3}}{q^2, q} \\
&=\fr{q^{2m}}{1-q^3}\sum_{n\geq 0} \fr{q^n \bigl(q^{2n+2},q^{2m+2n+3};q^2\bigr)_\infty}{\bigl(q^{2n+5},q^{2m+2n};q^2\bigr)_\infty} \\
&=\fr{1}{1-q^3}\sum_{n\geq 0}\fr{q^{2m+n}\bigl(q^{2n+2};q^2\bigr)_{m-1}}{\bigl(q^{2n+5};q^2\bigr)_{m-1}},
\end{align*}
where in the fourth step we applied~\eqref{Heine-1} with $(a,b,c,z)=\bigl(q^3,q,q^6,q^{2m}\bigr)$. This proves the desired result.

As for~\eqref{A'-2}, we omit the details as the proof follows exactly the same steps as in the proof of~\eqref{A'-1} with the exception that~\eqref{Heine-1}
is employed with $(a,b,c,z)=\bigl(q,q^3,q^6,q^{2m}\bigr)$ rather than~${(a,b,c,z)=\bigl(q^3,q,q^6,q^{2m}\bigr)}$.
\end{proof}
\begin{Lemma}\label{lem A'-A}
There holds
$
\sum_{n\geq 0} A'(m,n) q^n
=\sum_{n\geq 0} ( A(m,n)-A(m+2,n) ) q^n$.
\end{Lemma}
\begin{proof}
We need the following contiguous relation which can be found in~\cite[equation (2.1)]{Krattenthaler1993}:
\begin{equation}
\pFq{2}{1}{a,b}{c}{q, z} - \pFq{2}{1}{a,b}{c}{q, qz} = z \fr{(1-a)(1-b)}{1-c} \pFq{2}{1}{qa,qb}{qc}{q, z}.\label{Kratt-21}
\end{equation}
Then by~\eqref{Kratt-21} applied with $q\to q^2$ and $(a,b,c,z)=\bigl(q^{-1},q,q^4,q^m\bigr)$, we get
\begin{gather*}
\sum_{n\geq 0} ( A(m,n)-A(m+2,n) ) q^n\\
\qquad= \fr{\bigl(q^2,q^4;q^2\bigr)_\infty}{\bigl(q^{-1},q;q^2\bigr)_\infty}
\left( \pFq{2}{1}{q^{-1},q}{q^4}{q^2, q^m}
 - \pFq{2}{1}{q^{-1},q}{q^4}{q^2, q^{m+2}} \right) \\
\qquad=\fr{\bigl(q^2,q^4;q^2\bigr)_\infty}{\bigl(q^{-1},q;q^2\bigr)_\infty} q^m \fr{\bigl(1-q^{-1}\bigr)(1-q)}{1-q^4}
\pFq{2}{1}{q,q^3}{q^6}{q^2, q^{m}} \\
\qquad=\fr{\bigl(q^2,q^6;q^2\bigr)_\infty}{\bigl(q,q^3;q^2\bigr)_\infty} \sum_{n\geq 0}\fr{q^{m(n+1)} \bigl(q,q^3;q^2\bigr)_n}{\bigl(q^2,q^6;q^2\bigr)_n} =\sum_{n\geq 0}\fr{q^{m(n+1)}\bigl(q^{2n+2};q^2\bigr)_\infty \bigl(q^{2n+6};q^2\bigr)_\infty}{\bigl(q^{2n+1};q^2\bigr)_\infty \bigl(q^{2n+3};q^2\bigr)_\infty}.
\end{gather*}
This proves the lemma.
\end{proof}
\begin{Lemma}\label{lem B}
We have
\begin{equation}
\sum_{n\geq 0}B(2,n)q^n = \fr{-q^3\bigl(1-q^3-q^4-q^5\bigr)}{\bigl(1-q^3\bigr)^2 \bigl(1-q^5\bigr)}\label{B2}
\end{equation}
and for any positive integer $m\geq 2$,
\begin{equation}
\sum_{n\geq 0} B(2m,n) q^n
= \fr{1}{(1-q)\bigl(1-q^3\bigr)}\sum_{n\geq 0}\fr{q^{3n+4}\bigl(q^{2n+2};q^2\bigr)_{m-1}}{\bigl(q^{2n+1};q^2\bigr)_{m-2}}.\label{B2m}
\end{equation}
\end{Lemma}
By~\eqref{basic} and~\eqref{Heine-2}, we obtain
\begin{align*}
\sum_{n\geq 0}B(2,n)q^{2n} &=\sum_{n\geq 0} \fr{q^{2n} \bigl(q^{2n+2},q^{2n+4};q^2\bigr)_\infty}{\bigl(q^{2n-3},q^{2n+3};q^2\bigr)_\infty} =\fr{\bigl(q^2,q^4;q^2\bigr)_\infty}{\bigl(q^{-3},q^3;q^2\bigr)_\infty} \sum_{n\geq 0}\fr{\bigl(q^{-3},q^3;q^2\bigr)_n q^{2n}}{\bigl(q^2,q^4;q^2\bigr)_n} \\
&=\fr{\bigl(q^2,q^4;q^2\bigr)_\infty}{\bigl(q^{-3},q^3;q^2\bigr)_\infty}\pFq{2}{1}{q^{-3},q^3}{q^{4}}{q^2, q^{2}} \\
&=\fr{\bigl(q^2,q^4;q^2\bigr)_\infty}{\bigl(q^{-1},q;q^2\bigr)_\infty} \fr{\bigl(q,q^5;q^2\bigr)_\infty}{\bigl(q^2,q^4;q^2\bigr)_\infty}
\pFq{2}{1}{q^{-2},q^3}{q^{5}}{q^2, q} \\
&=\fr{q^4}{(1-q)\bigl(1-q^3\bigr)^2} \left(1+ \fr{q\bigl(1-q^{-2}\bigr)\bigl(1-q^3\bigr)}{\bigl(1-q^2\bigr)\bigl(1-q^5\bigr)} \right) \\
&=\fr{q^4}{(1-q)\bigl(1-q^3\bigr)^2} -\fr{q^3}{(1-q)\bigl(1-q^3\bigr)\bigl(1-q^5\bigr)} =\fr{-q^3\bigl(1-q^3-q^4-q^5\bigr)}{\bigl(1-q^3\bigr)^2 \bigl(1-q^5\bigr)}.
\end{align*}
This proves~\eqref{B2}.

Now assume that $m$ is a positive integer such that $m\geq 2$. Then
by using~\eqref{basic} and~\eqref{Heine-1}, we have
\begin{align*}
\sum_{n\geq 0} B(2m,n) q^n
&=\sum_{n\geq 0} \fr{q^{2mn} \bigl(q^{2n+2},q^{2n+4};q^2\bigr)_\infty}{\bigl(q^{2n-3},q^{2n+3};q^2\bigr)_\infty} \\
&=\fr{\bigl(q^2,q^4;q^2\bigr)_\infty}{\bigl(q^{-3},q^3;q^2\bigr)_\infty}
\pFq{2}{1}{q^{-3},q^3}{q^{4}}{q^2, q^{2m}} \\
&=\fr{\bigl(q^2,q^4;q^2\bigr)_\infty}{\bigl(q^{-3},q^3;q^2\bigr)_\infty}\fr{\bigl(q^3,q^{2m-3};q^2\bigr)_\infty}{\bigl(q^4,q^{2m};q^2\bigr)_\infty}
\pFq{2}{1}{q,q^{2m}}{q^{2m-3}}{q^2, q^3} \\
&=\fr{\bigl(q^2,q^{2m-3};q^2\bigr)_\infty}{\bigl(q^{-3},q^{2m};q^2\bigr)_\infty}\fr{\bigl(q,q^{2m};q^2\bigr)_\infty}{\bigl(q^2,q^{2m-3};q^2\bigr)_\infty}
\sum_{n\geq 0}\fr{q^{3n} \bigl(q^{2n+2}, q^{2n+2m-3};q^2\bigr)_\infty}{\bigl(q^{2n+1},q^{2n+2m};q^2\bigr)_\infty} \\
&=\fr{q^4}{(1-q)\bigl(1-q^3\bigr)}\sum_{n\geq 0}\fr{q^{3n} \bigl(q^{2n+2};q^2\bigr)_{m-1}}{\bigl(q^{2n+1};q^2\bigr)_{m-2}},
\end{align*}
which yields~\eqref{B2m}.
\begin{Lemma}\label{lem B'}
For any positive integer, we have
\[
\sum_{n\geq 0} B'(2m,n) q^n
=\fr{-1}{1-q}\sum_{n\geq 0}\fr{q^{5n+2m+1}\bigl(q^{2n+2};q^2\bigr)_{m-1}}{\bigl(q^{2n+1};q^2\bigr)_{m-1}}.
\]
\end{Lemma}
\begin{proof}
By~\eqref{basic} and~\eqref{Heine-1} applied to $(a,b,c,z)=\bigl(q^{-1},q^5,q^6,q^{2m}\bigr)$, we find
\begin{align*}
\sum_{n\geq 0} B'(2m,n) q^n
&=q^{2m} \fr{\bigl(q^2,q^6;q^2\bigr)_\infty}{\bigl(q^{-1},q^5;q^2\bigr)_\infty}
\pFq{2}{1}{q^{-1},q^5}{q^6}{q^2, q^{2m}} \\
&=q^{2m} \fr{\bigl(q^2,q^6;q^2\bigr)_\infty}{\bigl(q^{-1},q^5;q^2\bigr)_\infty}
\fr{\bigl(q^5,q^{2m-1};q^2\bigr)_\infty}{\bigl(q^6,q^{2m};q^2\bigr)_\infty}
\pFq{2}{1}{q,q^{2m}}{q^{2m-1}}{q^2, q^5} \\
&=q^{2m} \fr{\bigl(q^2,q^{2m-1};q^2\bigr)_\infty}{\bigl(q^{-1},q^{2m};q^2\bigr)_\infty}
\sum_{n\geq 0}\fr{q^{5n} \bigl(q,q^{2m};q^2\bigr)_n}{\bigl(q^2,q^{2m-1};q^2\bigr)_n} \\
&=\fr{-1}{1-q}\sum_{n\geq 0} \fr{q^{5n+2m+1} \bigl(q^{2n+2},q^{2n+2m-1};q^2\bigr)_\infty}{\bigl(q^{2n+1},q^{2n+2m};q^2\bigr)_\infty} \\
&=\fr{-1}{1-q}\sum_{n\geq 0}\fr{q^{5n+2m+1}\bigl(q^{2n+2};q^2\bigr)_{m-1}}{\bigl(q^{2n+1};q^2\bigr)_{m-1}},
\end{align*}
which is the desired identity.
\end{proof}
\begin{Lemma}\label{lem B'-B}
There holds
$
\sum_{n\geq 0} B'(m,n) q^n
=\sum_{n\geq 0} ( B(m,n)-B(m+2,n) ) q^n$.
\end{Lemma}
\begin{proof}
By~\eqref{Kratt-21} applied with $q\to q^2$ and $(a,b,c,z)=\bigl(q^{-3},q^3,q^4,q^m\bigr)$, we get
\begin{gather*}
\sum_{n\geq 0} ( B(m,n)-B(m+2,n) ) q^n\\
\qquad= \fr{\bigl(q^2,q^4;q^2\bigr)_\infty}{\bigl(q^{-3},q^3;q^2\bigr)_\infty}
\left( \pFq{2}{1}{q^{-3},q^3}{q^4}{q^2, q^m} - \pFq{2}{1}{q^{-3},^3}{q^4}{q^2, q^{m+2}} \right) \\
\qquad=\fr{\bigl(q^2,q^4;q^2\bigr)_\infty}{\bigl(q^{-3},q^3;q^2\bigr)_\infty} q^m \fr{\bigl(1-q^{-3}\bigr)\bigl(1-q^3\bigr)}{1-q^4}
\pFq{2}{1}{q^{-1},q^5}{q^6}{q^2, q^{m}} \\
\qquad=\fr{\bigl(q^2,q^6;q^2\bigr)_\infty}{\bigl(q^{-1},q^5;q^2\bigr)_\infty} \sum_{n\geq 0}\fr{q^{m(n+1)} \bigl(q^{-1},q^5;q^2\bigr)_n}{\bigl(q^2,q^6;q^2\bigr)_n} =\sum_{n\geq 0}\fr{q^{m(n+1)}\bigl(q^{2n+2};q^2\bigr)_\infty \bigl(q^{2n+6};q^2\bigr)_\infty}{\bigl(q^{2n-1};q^2\bigr)_\infty \bigl(q^{2n+5};q^2\bigr)_\infty},
\end{gather*}
as desired.
\end{proof}

\section[Proof of Theorem 2.2 and Corollary 2.3]{Proof of Theorem~\ref{thm double-1} and Corollary~\ref{cor positive-1}}\label{sec proof double-1}
\begin{proof}[Proof of Theorem~\ref{thm double-1}]
By virtue of Lemma~\ref{lem A'-A}, we obtain
%\[
\begin{align}
\sum_{n\geq 0} A(2m+2,n)q^n
&= \sum_{n\geq 0} A(2m,n)q^n - \sum_{n\geq 0} A'(2m,n)q^n \nonumber \\
&=\sum_{n\geq 0} A(2m-2,n)q^n - \bigg(\sum_{n\geq 0} A'(2m-2,n)q^n + \sum_{n\geq 0} A'(2m,n)q^n \bigg) \nonumber\\
& \quad\cdots \nonumber \\
&=\sum_{n\geq 0} A(2,n)q^n -\sum_{k=1}^m \sum_{n\geq 0}A'(2k,n)q^n \label{help-1 double-1}.
\end{align}
%\]
Now combine~\eqref{help-1 double-1} with Lemma~\ref{lem A} and~\eqref{A'-2} to deduce
\[
\fr{-1}{(1-q)^2}\sum_{n\geq 0} \fr{q^{n+1} \bigl(q^{2n+2};q^2\bigr)_m}{\bigl(q^{2n+3};q^2\bigr)_{m-1}}
=\fr{-q\bigl(1+q^2\bigr)}{(1-q)^2 \bigl(1-q^3\bigr)}-\fr{1}{1-q}\sum_{n\geq 0}\sum_{k=1}^m \fr{q^{2k+3n} \bigl(q^{2n+2};q^2\bigr)_{k-1}}{\bigl(q^{2n+3};q^2\bigr)_{k-1}},
\]
which by letting $m\to \infty$ yields
\begin{equation}
\fr{-q}{(1-q)^2}\sum_{n\geq 0} \fr{q^{n} \bigl(q^{2n+2};q^2\bigr)_\infty}{\bigl(q^{2n+3};q^2\bigr)_\infty}
=\fr{-q\bigl(1+q^2\bigr)}{(1-q)^2 \bigl(1-q^3\bigr)}-\fr{1}{1-q} \sum_{n,k \geq 0}\fr{q^{2k+5n+2} \bigl(q^{2n+2};q^2\bigr)_{k}}{\bigl(q^{2n+1};q^2\bigr)_{k}}.\label{help id-1}
\end{equation}
In addition, by~\eqref{q-binom}
and simplification we have
\begin{align}
\sum_{n\geq 0} \fr{q^{n} \bigl(q^{2n+2};q^2\bigr)_\infty}{\bigl(q^{2n+3};q^2\bigr)_\infty}
&=\fr{\bigl(q^2;q^2\bigr)_\infty}{\bigl(q^3;q^2\bigr)_\infty} \sum_{n\geq 0} \fr{\bigl(q^3;q^2\bigr)_n q^n}{\bigl(q^2;q^2\bigr)_n} =\fr{\bigl(q^2;q^2\bigr)_\infty}{\bigl(q^3;q^2\bigr)_\infty} \fr{\bigl(q^4;q^2\bigr)_\infty}{\bigl(q;q^2\bigr)_\infty} \nonumber \\
&= \fr{1-q}{1-q^2} \fr{\bigl(q^2;q^2\bigr)_\infty^2}{\bigl(q;q^2\bigr)_\infty^2}.\label{help-2 double-1}
\end{align}
Finally, combine~\eqref{help-2 double-1} with~\eqref{help id-1} and rearrange to obtain
\[
\sum_{n,k \geq 0}\fr{q^{2k+3n+2} \bigl(q^{2n+2};q^2\bigr)_{k}}{\bigl(q^{2n+3};q^2\bigr)_{k}}
=\fr{1}{1-q^2} \fr{\bigl(q^2;q^2\bigr)_\infty^2}{\bigl(q;q^2\bigr)_\infty^2} -\fr{1+q^2}{(1-q)\bigl(1-q^3\bigr)},
\]
which gives the desired formula.
\end{proof}

\begin{proof}[Proof of Corollary~\ref{cor positive-1}]
An application of~\eqref{psi-sum} with $q\to q^4$ and $(a,b,z)=\bigl(q,q^5,q\bigr)$ yields after simplification
\[
\sum_{n=-\infty}^{\infty}\fr{(q;q^4)_n q^n}{(q^5;q^4)_n}
= (1-q)\fr{\bigl(q^2;q^2\bigr)_\infty^2}{\bigl(q;q^2\bigr)_\infty^2},
\]
that is,
$
\sum_{n=-\infty}^{\infty}\fr{q^n}{1-q^{4n+1}} = \fr{(q^2;q^2)_\infty^2}{(q;q^2)_\infty^2}$,
or equivalently,
\begin{equation}
\sum_{n=0}^\infty \left( \fr{q^n}{1-q^{4n+1}}-\fr{q^{3n+2}}{1-q^{4n+3}}\right)
= \fr{\bigl(q^2;q^2\bigr)_\infty^2}{\bigl(q;q^2\bigr)_\infty^2}.\label{help-1 positive-1}
\end{equation}
Thus, by using Theorem~\ref{thm double-1} and~\eqref{help-1 positive-1} we obtain
\begin{align*}
\sum_{n\geq 0}F_1(n) q^n
={}&\fr{-\bigl(1+q^2\bigr)}{(1-q)\bigl(1-q^3\bigr)} + \fr{1}{1-q^2} \fr{\bigl(q^2;q^2\bigr)_\infty^2}{\bigl(q;q^2\bigr)_\infty^2} \\
={}&\fr{-\bigl(1+q^2\bigr)}{(1-q)\bigl(1-q^3\bigr)} +\fr{1}{1-q^2}\left(\fr{1}{1-q}-\fr{q^2}{1-q^3} \right)\\
&
+\sum_{n\geq 1}\left( \fr{q^n}{1-q^{4n+1}}-\fr{q^{3n+2}}{1-q^{4n+3}}\right)\fr{1}{1-q^2} \\
={}&\fr{-q^2}{(1-q)\bigl(1-q^3\bigr)}
+\sum_{n\geq 1}\left(\fr{q^n\bigl(1-q^{2n}\bigr)}{\bigl(1-q^{4n+1}\bigr)\bigl(1-q^2\bigr)}+\fr{q^{3n}}{\bigl(1-q^{4n+1}\bigr)\bigl(1-q^{4n+3}\bigr)}\right) \\
&\succeq \fr{-q^2}{(1-q)\bigl(1-q^3\bigr)} + \sum_{n\geq 1}\fr{q^n\bigl(1-q^{2n}\bigr)}{1-q^2} + \sum_{n\geq 1}q^{3n} \\
={}&\fr{-q^2}{(1-q)\bigl(1-q^3\bigr)} + \fr{q}{(1-q)\bigl(1-q^2\bigr)}-\fr{q^3}{\bigl(1-q^2\bigr)\bigl(1-q^3\bigr)}+\fr{q^3}{1-q^3} \\
={}&\fr{q+q^3}{1-q^3} \succeq 0.
\end{align*}
This completes the proof.
\end{proof}
\section[Proof of Theorem 2.5 and Corollary 2.6]{Proof of Theorem~\ref{thm double-2} and Corollary~\ref{cor positive-2}}\label{sec proof double-2}
\begin{proof}[Proof of Theorem~\ref{thm double-2}]
Combining~\eqref{help-1 double-1} with Lemma~\ref{lem A} and~\eqref{A'-1}, we derive
\[
\fr{1}{(1-q)^2}\sum_{n\geq 0}\fr{q^{n+1}\bigl(q^{2n+2};q^2\bigr)_m}{\bigl(q^{2n+3};q^2\bigr)_{m-1}}
=\fr{q\bigl(1+q^2\bigr)}{(1-q)^2 \bigl(1-q^3\bigr)}
+ \fr{1}{1-q^3}\sum_{n\geq 0}\sum_{k=1}^m \fr{q^{2k+n}\bigl(q^{2n+2};q^2\bigr)_{k-1}}{\bigl(q^{2n+5};q^2\bigr)_{k-1}},
\]
which by letting $m\to \infty$ implies
\[
\fr{1}{(1-q)^2}\sum_{n\geq 0}\fr{q^{n+1}\bigl(q^{2n+2};q^2\bigr)_\infty}{\bigl(q^{2n+3};q^2\bigr)_{\infty}}
=\fr{q\bigl(1+q^2\bigr)}{(1-q)^2 \bigl(1-q^3\bigr)}
+ \fr{1}{1-q^3}\sum_{n,k\geq 0}\fr{q^{2k+2+n}\bigl(q^{2n+2};q^2\bigr)_{k}}{\bigl(q^{2n+5};q^2\bigr)_{k}},
\]
or equivalently,
\[
\sum_{n,k\geq 0}\fr{q^{2k+2+n}\bigl(q^{2n+2};q^2\bigr)_{k}}{\bigl(q^{2n+5};q^2\bigr)_{k}}
=\fr{1-q^3}{(1-q)^2}\sum_{n\geq 0}\fr{q^{n+1}\bigl(q^{2n+2};q^2\bigr)_\infty}{\bigl(q^{2n+3};q^2\bigr)_{\infty}}
-\fr{q\bigl(1+q^2\bigr)}{(1-q)^2}.
\]
Now using~\eqref{help-2 double-1}, we derive
\[
\sum_{n,k\geq 0}\fr{q^{2k+2+n}\bigl(q^{2n+2};q^2\bigr)_{k}}{\bigl(q^{2n+5};q^2\bigr)_{k}}
=\fr{q\bigl(1-q^3\bigr)}{(1-q)^2}\fr{1-q}{1-q^2}\fr{\bigl(q^2;q^2\bigr)_\infty^2}{\bigl(q;;q^2\bigr)_\infty^2}
-\fr{q\bigl(1+q^2\bigr)}{(1-q)^2},
\]
which clearly is equivalent to the desired formula.
\end{proof}

\begin{proof}[Proof of Corollary~\ref{cor positive-2}]
From Theorem~\ref{thm double-2} and~\eqref{help-1 positive-1}, we obtain
\begin{align*}
\sum_{n\geq 0}F_2(n) q^n ={}&
\fr{-q\bigl(1+q^2\bigr)}{(1-q)^2} + \fr{q\bigl(1-q^3\bigr)}{(1-q)\bigl(1-q^2\bigr)}\fr{\bigl(q^2;q^2\bigr)_\infty^2}{\bigl(q;q^2\bigr)_\infty^2} \\
={}&\fr{-q\bigl(1+q^2\bigr)}{(1-q)^2}
+\fr{q\bigl(1-q^3\bigr)}{(1-q)\bigl(1-q^2\bigr)} \left( \fr{1-q^2}{(1-q)\bigl(1-q^3\bigr)}\right. \\
& \left.+\sum_{n\geq 1}\left(\fr{q^n}{1-q^{4n+1}}-\fr{q^{3n+2}}{1-q^{4n+3}} \right) \right) \\
={}&\fr{-q^3}{(1-q)^2} + \sum_{n\geq 1}\left(\fr{q^n\bigl(1-q^{2n}\bigr)}{\bigl(1-q^{4n+1}\bigr)\bigl(1-q^2\bigr)}+\fr{q^{3n}}{\bigl(1-q^{4n+1}\bigr)\bigl(1-q^{4n+3}\bigr)}\right) \\
\succeq{}& \fr{-q^3}{(1-q)^2}+\sum_{n\geq 1}\left( \fr{q^n\bigl(1-q^{2n}\bigr)}{1-q^2}+q^{3n} \right) \\
={}&\fr{-q^3}{(1-q)^2}+\fr{q}{(1-q)\bigl(1-q^2\bigr)}-\fr{q^3}{\bigl(1-q^2\bigr)\bigl(1-q^3\bigr)} + \fr{q^3}{1-q^3} \\
={}&\fr{-q^3}{(1-q)^2}+\fr{q}{(1-q)\bigl(1-q^2\bigr)} +\fr{q^3-q^4}{(1-q)\bigl(1-q^3\bigr)} \\
={}&\fr{2-2q^4+q^5}{(1-q)\bigl(1-q^3\bigr)}
=\fr{q\bigl(1-q^3\bigr)^2 +q^5-q^7}{(1-q)\bigl(1-q^3\bigr)} \\
={}&q\bigl(1+q+q^2\bigr) +\fr{q^5(1+q)}{1-q^3} \succeq 0.
\end{align*}
This gives the desired result.
\end{proof}
\section[Proof of Theorem~2.8 and Corollary 2.9]{Proof of Theorem~\ref{thm double-3} and Corollary~\ref{cor positive-3}}\label{sec proof double-3}
\begin{proof}[Proof of Theorem~\ref{thm double-3}]
By Lemma~\ref{lem B'-B}, we get
\begin{align}
\sum_{n\geq 0} B(2m+2,n)q^n
&= \sum_{n\geq 0} B(2m,n)q^n - \sum_{n\geq 0} B'(2m,n)q^n \nonumber \\
&=\sum_{n\geq 0} B(2m-2,n)q^n - \bigg(\sum_{n\geq 0} B'(2m-2,n)q^n + \sum_{n\geq 0} B'(2m,n)q^n \bigg) \nonumber\\
& \quad\cdots \nonumber \\
&=\sum_{n\geq 0} B(2,n)q^n -\sum_{k=1}^m \sum_{n\geq 0}B'(2k,n)q^n \label{help-1 double-3}.
\end{align}
Now use~\eqref{help-1 double-3}, Lemmas~\ref{lem B} and~\ref{lem B'} to deduce
\begin{gather*}
\fr{1}{(1-q)\bigl(1-q^3\bigr)}\sum_{n\geq 0} \fr{q^{3n+4} \bigl(q^{2n+2};q^2\bigr)_m}{\bigl(q^{2n+1};q^2\bigr)_{m-1}}
\\
\qquad=\fr{q^3\bigl(-1+q^3+q^4+q^5\bigr)}{\bigl(1-q^3\bigr)^2 \bigl(1-q^5\bigr)}+\fr{q}{1-q}\sum_{n\geq 0}\sum_{k=1}^m \fr{q^{2k+5n} \bigl(q^{2n+2};q^2\bigr)_{k-1}}{\bigl(q^{2n+1};q^2\bigr)_{k-1}},
\end{gather*}
which by letting $m\to \infty$ and using~\eqref{q-binom} gives
\begin{align*}
\sum_{n,k \geq 0}\fr{q^{2k+3n+2} \bigl(q^{2n+2};q^2\bigr)_{k}}{\bigl(q^{2n+3};q^2\bigr)_{k}}
&=\fr{q^3}{1-q^3}\sum_{n\geq 0} \fr{q^{3n} \bigl(q^{2n+2};q^2\bigr)_\infty}{\bigl(q^{2n+1};q^2\bigr)_\infty}
-\fr{q^2(1-q)\bigl(-1+q^3+q^4+q^5\bigr)}{\bigl(1-q^3\bigr)^2 \bigl(1-q^5\bigr)} \\
&=\fr{q^3}{1-q^3} \fr{\bigl(q^2;q^2\bigr)_\infty}{\bigl(q;q^2\bigr)_\infty}\fr{\bigl(q^4;q^2\bigr)_\infty}{\bigl(q^3;q^2\bigr)_\infty}
-\fr{q^2(1-q)\bigl(-1+q^3+q^4+q^5\bigr)}{\bigl(1-q^3\bigr)^2 \bigl(1-q^5\bigr)} \\
&=\fr{q^3(1-q)}{\bigl(1-q^2\bigr)\bigl(1-q^3\bigr)} \fr{\bigl(q^2;q^2\bigr)_\infty^2}{\bigl(q;q^2\bigr)_\infty^2}
-\fr{q^2(1-q)\bigl(-1+q^3+q^4+q^5\bigr)}{\bigl(1-q^3\bigr)^2 \bigl(1-q^5\bigr)} \\
&=\fr{q^3}{(1+q)\bigl(1-q^3\bigr)} \fr{\bigl(q^2;q^2\bigr)_\infty^2}{\bigl(q;q^2\bigr)_\infty^2}
-\fr{q^2(1-q)\bigl(-1+q^3+q^4+q^5\bigr)}{\bigl(1-q^3\bigr)^2 \bigl(1-q^5\bigr)},
\end{align*}
which is the desired identity.
\end{proof}

\begin{proof}[Proof of Corollary~\ref{cor positive-3}]
By virtue of Theorem~\ref{thm double-3} and~\eqref{help-1 positive-1}, we have
\begin{align*}
\sum_{n\geq 0} G(n) q^n
={}&\fr{q^3}{(1+q)\bigl(1-q^3\bigr)} \fr{\bigl(q^2;q^2\bigr)_\infty^2}{\bigl(q;q^2\bigr)_\infty^2}
-\fr{q^2(1-q)\bigl(-1+q^3+q^4+q^5\bigr)}{\bigl(1-q^3\bigr)^2 \bigl(1-q^5\bigr)} \\
={}&\fr{q^3}{(1+q)\bigl(1-q^3\bigr)} \left(\fr{1}{1-q}-\fr{q^2}{1-q^3}\right)\\
&+\fr{q^3}{(1+q)\bigl(1-q^3\bigr)}\sum_{n\geq 1}\left(\fr{q^n}{1-q^{4n+1}}-\fr{q^{3n+2}}{1-q^{4n+3}}\right) \\
& -\fr{q^2(1-q)\bigl(-1+q^3+q^4+q^5\bigr)}{\bigl(1-q^3\bigr)^2 \bigl(1-q^5\bigr)} \\
={}&\fr{q^2}{\bigl(1-q^3\bigr)\bigl(1-q^5\bigr)} + \sum_{n\geq 1}\fr{q^{n+3}-q^{3n+5}-q^{5n+6}+q^{7n+6}}{(1+q)\bigl(1-q^3\bigr)\bigl(1-q^{4n+1}\bigr)\bigl(1-q^{4n+3}\bigr)} \\
={}&\fr{q^2}{\bigl(1-q^3\bigr)\bigl(1-q^5\bigr)} + \sum_{n\geq 1}\fr{q^{n+3}\bigl(1-q^{2n+2}\bigr)-q^{5n+6}\bigl(1-q^{2n}\bigr)}{(1+q)\bigl(1-q^3\bigr)\bigl(1-q^{4n+1}\bigr)\bigl(1-q^{4n+3}\bigr)} \\
\succeq{}& \fr{q^2}{\bigl(1-q^3\bigr)\bigl(1-q^5\bigr)}+ \sum_{n\geq 1}\fr{q^{n+3}\bigl(1-q^{2n+2}\bigr)}{(1+q)\bigl(1-q^3\bigr)}
-\sum_{n\geq 1}\fr{q^{5n+6}\bigl(1-q^{2n}\bigr)}{(1+q)\bigl(1-q^3\bigr)} \\
={}&\fr{q^2}{\bigl(1-q^3\bigr)\bigl(1-q^5\bigr)}+ \fr{q^3}{(1+q)\bigl(1-q^3\bigr)}\left(\fr{1}{1-q}-\fr{q^2}{1-q^3}\right) \\
& +\fr{q^6}{(1+q)\bigl(1-q^3\bigr)}\left(\fr{-1}{1-q^5}+\fr{1}{1-q^7}\right) \\
={}&\fr{q^2}{\bigl(1-q^3\bigr)\bigl(1-q^5\bigr)}+ \fr{q^3}{\bigl(1-q^3\bigr)^2}-\fr{q^{11}(1-q)}{\bigl(1-q^3\bigr)\bigl(1-q^5\bigr)\bigl(1-q^7\bigr)} \\
\succeq{}& \fr{q^2}{\bigl(1-q^3\bigr)\bigl(1-q^5\bigr)}-\fr{q^{11}(1-q)}{\bigl(1-q^3\bigr)\bigl(1-q^5\bigr)\bigl(1-q^7\bigr)} \\
={}&\fr{q^2-q^9-q^{11}+q^{12}}{\bigl(1-q^3\bigr)\bigl(1-q^5\bigr)\bigl(1-q^7\bigr)}
=\fr{q^2\bigl(1-q^3\bigr)\bigl(1-q^7\bigr)+q^5-q^{11}}{\bigl(1-q^3\bigr)\bigl(1-q^5\bigr)\bigl(1-q^7\bigr)} \\
={}&\fr{q^2}{1-q^5}+\fr{q^5\bigl(1+q^3\bigr)}{\bigl(1-q^5\bigr)\bigl(1-q^7\bigr)}
=\fr{q^2}{1-q^5}+\fr{q^5\bigl(1-q^6\bigr)}{\bigl(1-q^3\bigr)\bigl(1-q^5\bigr)\bigl(1-q^7\bigr)} \succeq 0,
\end{align*}
as desired.
\end{proof}

\subsection*{Acknowledgements}
First author partially supported by Simons Foundation Grant 633284.
The authors are grateful to two anonymous referees for their valuable comments and interesting suggestions which have improved the presentation and quality of the paper.

\pdfbookmark[1]{References}{ref}
\LastPageEnding

\end{document}